\documentclass{amsart}

\usepackage[english]{babel}

\usepackage[letterpaper,top=2cm,bottom=2cm,left=2cm,right=2cm]{geometry}

\usepackage[style=alphabetic,doi=false,isbn=false,url=false,eprint=true,
backend=biber,giveninits=true,maxcitenames=3,maxbibnames=8,hyperref]{biblatex}
\usepackage{amsmath}
\usepackage{amsfonts}
\usepackage{amssymb}
\usepackage{amsthm}
\usepackage{graphicx}
\usepackage[colorlinks=true, allcolors=blue]{hyperref}
\usepackage{xcolor}
\usepackage{subcaption}
\usepackage{todonotes}
\usepackage{cleveref}
\usepackage{tikz-cd}
\usetikzlibrary{calc}

\usepackage[draft=true]{minted} 
\setminted{breaklines=true, breakanywhere=true}
\definecolor{lbcolor}{rgb}{0.9,0.9,0.9}
\setminted{bgcolor=lbcolor}
\setminted{fontsize=\footnotesize}

\numberwithin{equation}{section}

\theoremstyle{plain}
\newtheorem{lemma}[equation]{Lemma}
\newtheorem{theorem}[equation]{Theorem}

\newtheorem{proposition}[equation]{Proposition}
\theoremstyle{definition}
\newtheorem{definition}[equation]{Definition}
\newtheorem{remark}[equation]{Remark}

\newtheorem{example}[equation]{Example}

\title{Absolute incidence theorems and tilings}

\author[L. K\"uhne]{Lukas K\"uhne}
\address{Fakult\"at f\"ur Mathematik, Universit\"at Bielefeld, Bielefeld, Germany}
\email{\url{lkuehne@math.uni-bielefeld.de}}

\author[M. Larson]{Matt Larson}
\address{Institute for Advanced Study and Princeton University, NJ, USA}
\email{\url{mattlarson@princeton.edu}}

\begin{document}

\begin{abstract}
We give a precise definition of incidence theorems in plane projective geometry and introduce the notion of ``absolute incidence theorems,'' which hold over any \emph{ring}. 
Fomin and Pylyavskyy describe how to obtain incidence theorems from tilings of an orientable surface; they call this result  the ``master theorem''.
Instances of the master theorem are always absolute incidence theorems.
As most classically known incidence theorems are instances of the master theorem, they are absolute incidence theorems.
We give an explicit example of an incidence theorem involving 13 points that is not an absolute incidence theorem, and therefore is not an instance of the master theorem.
\end{abstract}

\maketitle

\vspace{-3mm}
\section{Introduction}

Incidence theorems in the projective plane are a classically studied subject, dating back to the ancient Greeks. In this paper, we study different types of incidence theorems. We begin by giving a precise definition of incidence theorems, in a sense that would be recognizable to Euclid. 

\begin{definition}\label{def:incidence}
An incidence theorem is a set $\{1, \dotsc, n\}$, a collection of \emph{nondegeneracy} conditions consisting of pairs and triples of elements of $\{1, \dotsc, n\}$, a collection of \emph{collinearity} conditions consisting of triples of elements of $\{1, \dotsc, n\}$, and an additional triple  $\{a, b, c\}$ called the \emph{conclusion}, satisfying the following condition: for every field $\mathbf{k}$ and points $p_1, \dotsc, p_n$ in $\mathbb{P}^2(\mathbf{k})$ such that
\begin{enumerate}
\item if $\{i, j\}$ is a nondegeneracy condition, then $p_i \not= p_j$, 
\item if $\{i, j, k\}$ is a nondegeneracy condition, then $p_i, p_j,$ and $p_k$ do not lie on a line, and
\item if $\{i, j, k\}$ is a collinearity condition, then $p_i, p_j,$ and $p_k$ lie on a line,
\end{enumerate} 
then $p_a, p_b$, and $p_c$ lie on a line. 
\end{definition}

Often, one assumes that all points are distinct. One sometimes considers incidence theorems which are valid over a particular field $\mathbf{k}$, meaning that one requires the above conditions only for points in $\mathbb{P}^2(\mathbf{k})$. For example, the ancient Greeks were interested in the case $\mathbf{k} = \mathbb{R}$. 
We can describe Pappus's theorem in this way.

\begin{example}
	Consider the configuration of points in~\Cref{fig:pappus}.
	This figure describes an incidence theorem in the sense of~\Cref{def:incidence}:
	\begin{enumerate}
		\item Assume that all pairs of points are distinct, that is, all pairs $\{i,j\}$ are nondegenerate.
		\item  The lines shown in~\Cref{fig:pappus} encode the collinearity conditions:
		\[
			\{1,2,3\},\, \{1,5,7\},\,\{1,6,8\},\,\{2,4,7\},\,\{2,6,9\},\,\{3,4,8\},\,\{3,5,9\},\,\{4,5,6\}.
		\]
		\item Lastly, we assume that all triples apart from the triple $\{7,8,9\}$  and these collinearity conditions are nondegenerate.
	\end{enumerate}
	Pappus's theorem now states that, in this situation, the additional incidence $\{7,8,9\}$ also holds, that is, for every choice of points $p_1,\dots,p_9$ in $\mathbb{P}^2(\mathbf{k})$ satisfying these conditions,  the points $p_7$, $p_8$ and $p_9$ are collinear.
	\begin{center}
		\begin{figure}[hbt]
		\begin{tikzpicture}[scale=.7]
			
			\coordinate (A) at (0,0);
			\coordinate (B) at (2,0);
			\coordinate (C) at (4,0);
			
			\coordinate (D) at (1,2);
			\coordinate (E) at (2.5,2);
			\coordinate (F) at (5,2);
			
			\draw[thick] (A) -- (C); 
			\draw[thick] (D) -- (F);
			
			\draw[thick] (A) -- (E);
			\draw[thick] (A) -- (F);
			\draw[thick] (B) -- (D);
			\draw[thick] (B) -- (F);
			\draw[thick] (C) -- (D);
			\draw[thick] (C) -- (E);
			
			\coordinate (P) at (intersection of A--E and B--D);
			\coordinate (Q) at (intersection of A--F and C--D);
			\coordinate (R) at (intersection of B--F and C--E);
			
			\fill[black] (P) circle (1.5pt) node[label={[xshift=-0.3cm, yshift=-0.38cm]$7$}] {};
			\fill[black] (Q) circle (1.5pt) node[above] {$8$};
			\fill[black] (R) circle (1.5pt) node[label={[xshift=0.3cm, yshift=-0.38cm]$9$}] {};
			
			\fill[black] (A) circle (1.5pt) node[below] {$1$};
			\fill[black] (B) circle (1.5pt) node[below] {$2$};
			\fill[black] (C) circle (1.5pt) node[below] {$3$};
			
			\fill[black] (D) circle (1.5pt) node[above] {$4$};
			\fill[black] (E) circle (1.5pt) node[above] {$5$};
			\fill[black] (F) circle (1.5pt) node[above] {$6$};
			
		\end{tikzpicture}
		\caption{The configuration of points and lines in Pappus's theorem.}
		\label{fig:pappus}
			\end{figure}
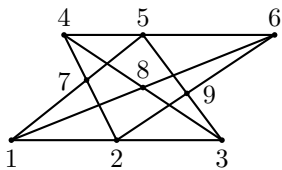
	\end{center}
\end{example}

We will also be interested in a stronger form of incidence theorem. For this, it will be convenient to recall that incidence theorems can be described algebraically. Given points $p_1, \dotsc, p_n$ in $\mathbb{P}^2(\mathbf{k})$, by choosing homogeneous coordinates, we can construct a $3 \times n$ matrix with entries in $\mathbf{k}$. Conversely, every $3 \times n$ matrix with no zero columns gives points $p_1, \dotsc, p_n$ in $\mathbb{P}^2(\mathbf{k})$. Any matrix that can be obtained by rescaling the columns gives the same points. 

In this framework, incidence theorems can be interpreted as algebraic statements. We have $p_i = p_j$ if and only if the $i$th column and the $j$th column are scalar multiples of each other, which is equivalent to the vanishing of the three $2 \times 2$ minors using the $i$th and $j$th columns of the matrix. Points $p_i, p_j$, and $p_k$ are collinear if and only if the corresponding $3 \times 3$ minor vanishes. Importantly, this perspective allows one to consider the validity of incidence theorems over commutative rings. 

\begin{definition}\label{def:absolute}
An incidence theorem with conclusion $\{a, b, c\}$ is an \emph{absolute incidence theorem} if, for every commutative ring $A$ and every $3 \times n$ matrix of elements of $A$ such that 
\begin{enumerate}
\item the ideal generated by the entries in each column is the unit ideal,
\item if $\{i, j\}$ is a nondegeneracy condition, then the ideal generated by the $2 \times 2$ minors using columns $i$ and $j$ is the unit ideal,
\item if $\{i, j, k\}$ is a nondegeneracy condition, then the determinant of the $3 \times 3$ minor with columns $i, j$, and $k$ is a unit, and
\item if $\{i, j, k\}$ is a collinearity condition, then the determinant of the $3 \times 3$ minor with columns $i, j$, and $k$ is $0$, 
\end{enumerate}
then the determinant of the $3 \times 3$ minor with columns $a, b$, and $c$ is $0$. 
\end{definition}

If $A$ is a field, then the conditions in Definition~\ref{def:absolute} reduce to the conditions in Definition~\ref{def:incidence}; the first condition is necessary because the vector $(0,0,0)$ does not define a point of projective space.  See Section~\ref{sec:absolute} for a discussion and basic properties of this definition.

We will show that most classically known incidence theorems are absolute incidence theorems. Our interest in incidence theorems comes from a recent paper of Fomin and Pylyavskyy \cite{FominPylyavskyy}, where they introduce a very general result which they call the ``master theorem,'' and they show that most or all classically known incidence theorems are instances of it. We will show that their techniques prove that these are absolute incidence theorems. 

To construct incidence theorems, Fomin and Pylyavskyy begin with a tiling of a closed orientable surface by quadrilaterals, where each vertex of the tiling is colored black or white, and each edge connects vertices of opposite colors. Choose a field $\mathbf{k}$, associate to each black vertex a point of $\mathbb{P}^2(\mathbf{k})$, and associate to each white vertex a line in $\mathbb{P}^2(\mathbf{k})$. We say that this tile is \emph{coherent} if neither point lies on either line, and either the points are the same, the lines are the same, or the intersection of the lines is contained in the line through the points. 

\begin{theorem}[Master theorem]\cite[Theorem 2.6]{FominPylyavskyy}\label{thm:master}
Suppose we have a tiling of a closed orientable surface and have associated a point to each black vertex and a line to each white vertex in $\mathbb{P}^2(\mathbf{k})$. If all but one of the tiles is coherent, then the last one is coherent as well. 
\end{theorem}

We now describe how to formulate a version of the master theorem in our framework. Given a tiling of a closed orientable surface by quadrilaterals, we choose an equivalence relation on the white vertices. We will associate a line to each equivalence class of white vertices, and we will require these lines to be distinct. 

Form the set which consists of the black vertices, two elements $s_i, t_i$ for each equivalence class of white vertices, and one element $R_k$ for each pair of distinct equivalence classes of white vertices which appear in a tile together. The nondegeneracy conditions are $\{s_i, t_i\}$ for each equivalence class of white vertices, $\{s_i, t_i, t_j\}, \{s_j, t_i, t_j\}, \{s_i, s_j, t_i\}$, and $\{s_i, s_j, t_j\}$ for each pair of equivalence classes of white vertices, and $\{s_i, t_i, P\}$ for each equivalence class of white vertices and black vertex $P$ which appear in a tile together. For each pair of distinct equivalence classes of white vertices which appear in a tile together, impose the collinearity conditions $\{s_i, t_i, R_k\}$ and $\{s_j, t_j, R_k\}$. 
To obtain an incidence theorem, choose one tile where the white vertices are not equivalent. For every other tile where the white vertices are not equivalent, impose the collinearity condition $\{P, Q, R_k\}$, where $P$ and $Q$ are the black vertices of the tile, and $R_k$ is the point constructed for the equivalence classes of white vertices in this tile. The conclusion is $\{P, Q, R_k\}$, where $P$ and $Q$ are the black vertices of the remaining tile where the white vertices are not equivalent, and $R_k$ is the point associated to this tile. We call this the incidence theorem generated by the tiling.

\begin{theorem}\label{thm:absolutemaster}
Given a tiling of a closed orientable surface by quadrilaterals and an equivalence relation on the white vertices, the incidence theorem generated by the tiling is an absolute incidence theorem. 
\end{theorem}

We illustrate the above setup in Example~\ref{ex:pappus}. The equivalence relation on the white vertices is necessary to formulate Theorem~\ref{thm:master} in our framework, as we cannot discuss lines directly. Instead, we construct two distinct points and think about the line that they determine. Our nondegeneracy conditions guarantee that these lines are distinct. 
In our framework, the basic identity underlying the master theorem (Proposition~\ref{prop:fundamental}) does not hold if the two lines appearing in the tile are the same, so we need to treat tiles where the two lines are the same in a different way. For this reason, one needs to specify in advance which lines are the same, which is why we need the equivalence relation on the white vertices. It can be beneficial to have a nontrivial equivalence relation, i.e., to identify some of the lines. See \cite[Remark 3.7]{FominPylyavskyy}. 

We prove Theorem~\ref{thm:absolutemaster} in Section~\ref{sec:master_theorem}. The proof relies on the same basic observation that Fomin and Pylyavskyy used to prove Theorem~\ref{thm:master}. However, there are a few technical obstacles that must be overcome, as we do not have access to tools from linear algebra because we work over arbitrary commutative rings.

We remark that a similar tiling-based approach to incidence theorems was introduced in \cite{RG}, see \cite{Baralic,RG2} for further work on this approach, and see \cite{PS} for more discussion and related approaches to incidence theorems.

Fomin and Pylyavskyy ask if every incidence theorem can be deduced from the master theorem. Pylyavskyy and Skopenkov gave an example of an incidence theorem which holds when $\mathbf{k} = \mathbb{C}$, but not when $\mathbf{k}$ has characteristic $2$ \cite[Example 3.1]{PS}. In particular, this result cannot be proven using the master theorem. We give an example of an incidence theorem in the sense of Definition~\ref{def:incidence} (so it holds over every field) which is not an absolute incidence theorem, and so it is not an instance of the master theorem.
The question of whether every absolute incidence theorem is an instance of the master theorem remains open.

\begin{theorem}\label{thm:incidencestrong}
Let $\mathbf{k}$ be a field, and suppose we have distinct points $p_1, \dotsc, p_{13}$ in the projective plane over $\mathbf{k}$.
Suppose that the following triples of points are collinear:
\begin{equation*}\begin{split}
&\{p_1, p_2, p_3\}, \{p_1, p_2, p_{13}\}, \{p_1, p_3, p_{13}\}, \{p_1, p_4, p_5\}, \{p_1, p_6, p_9\}, \{p_1, p_7, p_{10}\}, \{p_1, p_8, p_{12}\}, \{p_2, p_3, p_{13}\}, \\ 
&\{p_2, p_4, p_6\}, \{p_2, p_5, p_8\}, \{p_2, p_{10}, p_{11}\}, \{p_3, p_4, p_7\}, \{p_3, p_5, p_6\}, \{p_3, p_8, p_{10}\}, \{p_4, p_9, p_{10}\}, \{p_5, p_7, p_{11}\}, \\ 
&\{p_6, p_{7}, p_{13}\}, \{p_6, p_8, p_{11}\}, \{p_6, p_{10}, p_{12}\}, \{p_7, p_9, p_{12}\}. 
\end{split}\end{equation*}
Then $p_{11}, p_{12}$, and $p_{13}$ are collinear. Thus, this is an incidence theorem. However, it is not an absolute incidence theorem. 
\end{theorem}

The collinearity conditions are depicted in Figure~\ref{fig:M}. This incidence theorem holds over any field $\mathbf{k}$.
\begin{center}
	\begin{figure}[hbt]
		\begin{subfigure}{.48\linewidth}
			\includegraphics[width=\linewidth]{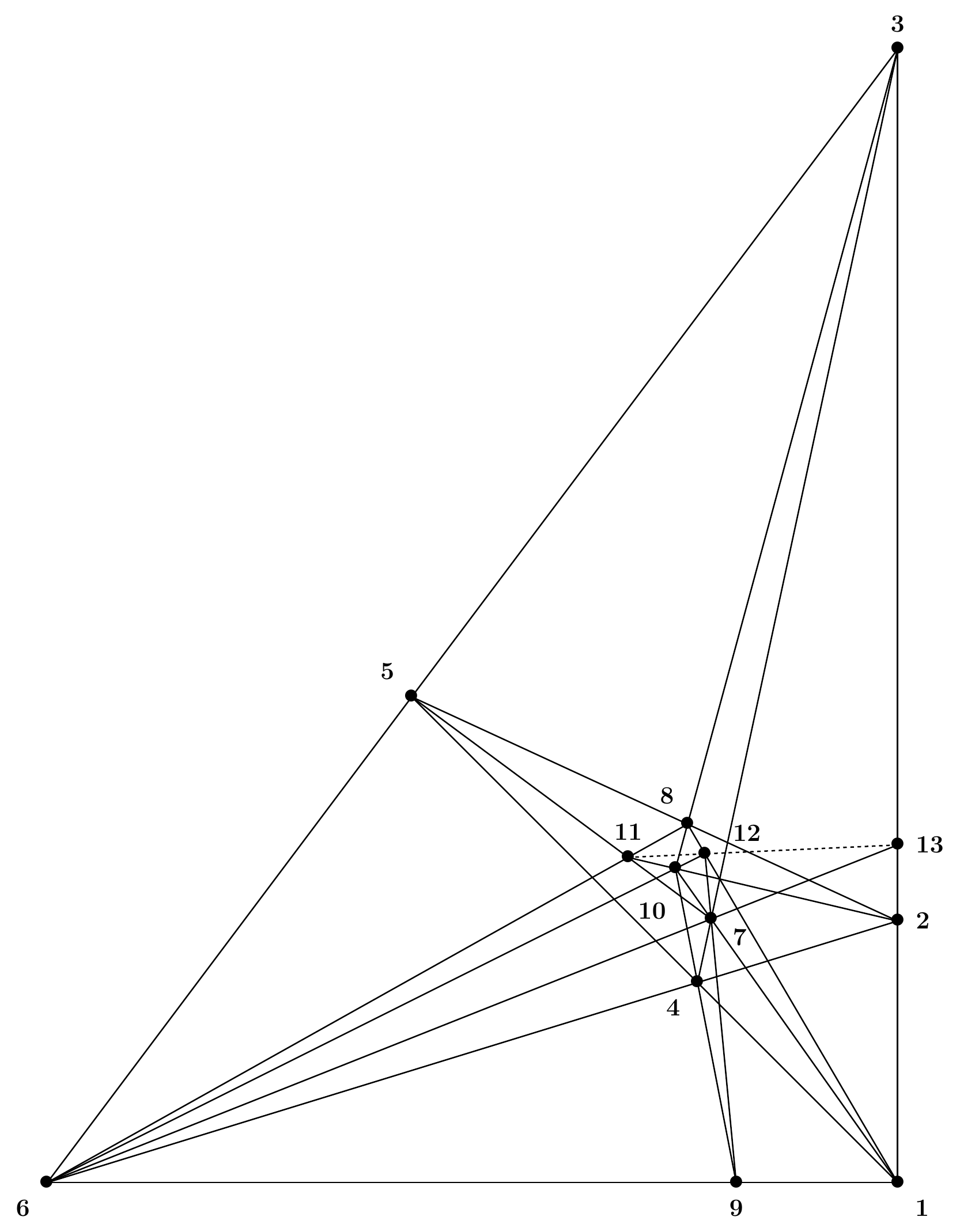}
			\caption{The collinearity conditions described in Theorem~\ref{thm:incidencestrong}.}
			\label{fig:M}
		\end{subfigure}
		\begin{subfigure}{.48\linewidth}
			\includegraphics[width=\linewidth]{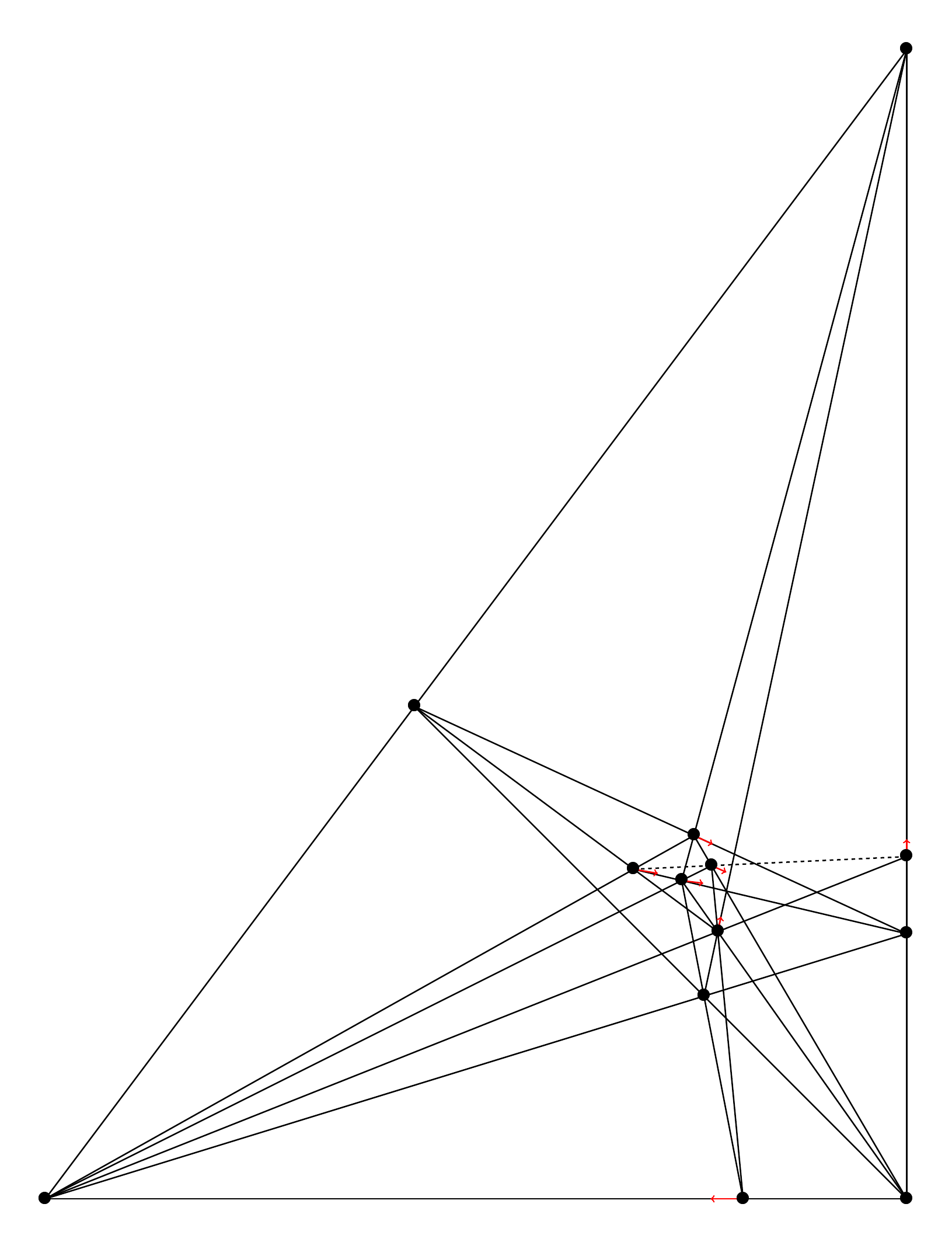}
			\caption{The infinitesimal motion of that configuration.}
			\label{fig:M_arrows}
		\end{subfigure}
		\caption{A configuration satisfying the conditions in Theorem~\ref{thm:incidencestrong} and its infinitesimal motion.}
	\end{figure}
\end{center}

We give two proofs of Theorem~\ref{thm:incidencestrong} in Section~\ref{sec:incidenceproof}. One proof is computer-assisted. The other proof is human-readable, but it involves a few tedious computations. 

To check that the incidence theorem is not an absolute incidence theorem, one only needs to give a $3 \times 13$ matrix with entries in a commutative ring $A$ such that the specified nondegeneracy and collinearity conditions hold, but the determinant of the $3 \times 3$ minor with columns $11, 12,$ and $13$ is not zero. If we set $A = \mathbb{Q}[x, \varepsilon]/(\varepsilon^2, x^2 - 2 - \varepsilon/4)$, then the following matrix suffices.
For instance, the determinant of the minor with columns $11, 12$, and $13$ is $2 \varepsilon/7 - \varepsilon x/14$. 

\begin{equation}\label{eq:matrix}
\left(
\begin{array}{cccccc}
1 & 1 & 0 & 1 & 0 & 0 \\
0 & 1 & 1 & 1 & 1 & 0 \\
0 & 0 & 0 & 1 & 1 & 1
\end{array}
\right.
\end{equation}

\vspace{-1.2em}

\[
\left.
\begin{array}{cccccccc}
1 & 1 & 1 & 1 & 1 & 1 & 1 \\
x +\frac{\varepsilon}{14} - \frac{\varepsilon x}{56} &
1 + x - \frac{3 \varepsilon}{28} - \frac{9 \varepsilon x}{56} &
0 &
2- \frac{3 \varepsilon}{28} - \frac{\varepsilon x}{28} &
1 + x - \frac{3 \varepsilon}{28} - \frac{9 \varepsilon x}{56} &
2- \frac{3 \varepsilon}{28} - \frac{\varepsilon x}{28} &
x +\frac{\varepsilon}{14} - \frac{\varepsilon x}{56} \\
1 &
x - \frac{3 \varepsilon}{28} - \frac{9 \varepsilon x}{56} &
2 - x + \frac{\varepsilon}{7} + \frac{5 \varepsilon x}{56} &
x - \frac{3 \varepsilon}{28} - \frac{9 \varepsilon x}{56} &
2 - \frac{5 \varepsilon}{28} - \frac{\varepsilon x}{7} &
4 - 2x &
0
\end{array}
\right).
\]

Incidence theorems over skew fields have been classically studied at least since the days of Hilbert \cite{Hilbert}. Unlike the situation over fields, it is undecidable to determine whether an incidence theorem holds over all skew fields \cite{KPY}. We do not know if the incidence theorem in Theorem~\ref{thm:incidencestrong} holds over skew fields. 

\medskip

We now describe the origin of the incidence theorem in Theorem~\ref{thm:incidencestrong}.  Let $M$ be the matroid obtained from \eqref{eq:matrix} by setting $\varepsilon = 0$, and let $M_0$ be the matroid obtained from $M$ by deleting the element $13$.  The rank $3$ simple matroid $M_0$ on $\{1, \dotsc, 12\}$ was shown to us by Dante Luber, who found it in his study with Dan Corey of the realization spaces of rank $3$ matroids on $12$ elements \cite{CL25}. It can be checked, for example using the \texttt{OSCAR} package for matroids~\cite{OscarMatroids}, that the rational realization space of $M_0$ (i.e., the quotient of the locus in $\operatorname{Gr}(3, 12)$ of linear subspaces realizing it by the action of the torus) is isomorphic to $\operatorname{Spec} \mathbb{Q}[x]/(x^2 - 2)^2$, i.e., to two non-reduced points. This means that there is an infinitesimal motion of the points in each realization of $M_0$ which preserves the incidence conditions, but this infinitesimal motion cannot be integrated. See Remark~\ref{rem:powerseries}. The matroid $M$ is obtained by adding a $13$th point at the intersection of the lines through $\{p_1, p_3\}$ and through $\{p_6, p_7\}$. In any of the realizations of $M_0$, the intersection of these lines is contained in the line passing through $\{p_{11}, p_{12}\}$. However, this fails infinitesimally: the infinitesimal motion of the points $p_1, \dotsc, p_{12}$ induces an infinitesimal motion of $p_{13}$ (as it is the intersection of the lines through $\{p_1, p_3\}$ and through $\{p_6, p_7\}$), but this motion does not preserve the coincidence that $p_{13}$ lies on the line through $\{p_{11}, p_{12}\}$.
The infinitesimal motion of these points is depicted with red arrows in Figure~\ref{fig:M_arrows}.
 The rational realization space of $M$ is isomorphic to $\operatorname{Spec} \mathbb{Q}[x]/(x^2 - 2)$.

\begin{remark}
The matroid $M_0$, found by Dan Corey and Dante Luber, is the first explicit example of a matroid whose realization space over a field of characteristic $0$ is non-reduced. The first example of a rank $3$ matroid with a singular realization space also occurs on $12$ elements \cite{CL25}. There are smaller examples of matroids whose realization space over a field of positive characteristic is non-reduced; for example, this occurs for the $9$-element Perles matroid in characteristic $5$. 
\end{remark}

\begin{remark}
Mn\"{e}v--Sturmfels universality \cite{Mnev,MR888886,Bokowski,MR1002208}, see also \cite{Lafforgue,LeeVakil}, shows that spaces defined by incidence equations (i.e., the subschemes of $\mathbb{A}^{3n}_{\mathbb{Z}}$ where certain nondegeneracy and collinearity conditions hold) can be almost arbitrary. However, the proofs of this result (such as the one in \cite{Lafforgue}) do not immediately adapt to produce incidence theorems which are not absolute incidence theorems. 
\end{remark}

\begin{remark}
In \cite{PS}, Pylyavskyy and Skopenkov gave a different definition of incidence theorems, in terms of incidence conditions on points and lines in the projective plane. Incidence theorems in this framework can be formulated in our framework and vice versa. One can similarly define absolute incidence theorems in their framework, and one can translate Theorem~\ref{thm:incidencestrong} into their framework to obtain an incidence theorem which is not an absolute incidence theorem. 
\end{remark}

\noindent
\textbf{Acknowledgments}
We thank Sergey Fomin for explaining to us several aspects of the beautiful work \cite{FominPylyavskyy}. We thank Kristof B\'erczi, Oliver Lorscheid and Tam\'as Schwarcz for help with Remark~\ref{rem:Nrealizable}. We thank Spencer Dembner, Sergey Fomin, Oliver Lorscheid, Pavlo Pylyavskyy, Mikhail Skopenkov, Yuchong Zhang, and the referee for helpful comments on a previous version of this paper. This work was partially conducted while the authors were visiting the Institute for Advanced Study, where the first author was funded by the Erik Ellentuck Fund and the second author is supported by the Charles Simonyi Endowment and the Oswald Veblen Fund.
The first author is supported by the Deutsche Forschungsgemeinschaft (DFG, German Research Foundation) -- SFB-TRR 358/1 2023 -- 491392403 and SPP 2458 -- 539866293.

\section{Absolute incidence theorems}\label{sec:absolute}

In this section, we discuss basic properties of absolute incidence theorems and make some remarks. We also give a direct proof that Pappus's theorem is an absolute incidence theorem. 
Throughout this section, we fix a commutative ring $A$.  We first give an alternative description of the nondegeneracy conditions in an absolute incidence theorem. 

\begin{lemma}\label{lem:field}
Elements $x, y, z, \dotsc$ of $A^3$ satisfy the nondegeneracy conditions of an absolute incidence theorem if and only if, for every maximal ideal $\mathfrak{m}$ of $A$, the images of $x, y, z,\dotsc$ in the vector space $(A/\mathfrak{m})^3$ over the field $A/\mathfrak{m}$ satisfy the nondegeneracy condition. 
\end{lemma}

\begin{proof}
The nondegeneracy conditions state that certain ideals are the unit ideal. In any commutative ring, every ideal except for the unit ideal is contained in a maximal ideal, so an ideal $I$ is the unit ideal if and only if there is no maximal ideal $\mathfrak{m}$ such that the image of $I$ in $A/\mathfrak{m}$ is $0$. 
\end{proof}

\begin{remark}\label{rem:integral}
If $\{x, y, z, \dotsc\}$ is a collection of elements of $A^3$ which satisfies conditions (1), (2), and (4) of Definition~\ref{def:absolute}, and the determinant of a $3 \times 3$ matrix corresponding to a nondegeneracy condition $\{i, j, k\}$ is not a zero-divisor, then the conclusion of the absolute incidence theorem will still hold. Indeed, we can consider the total ring of fractions $Q(A)$, which is the localization of $A$ at the set of all elements which are not zero-divisors. The natural map $A \to Q(A)$ is an injection, and the nondegeneracy conditions will be satisfied if we view this matrix as living in $Q(A)$. 
In particular, an incidence theorem (which holds over all fields) also holds over all integral domains. 
\end{remark}

\begin{remark}
Suppose that $\{x, y, z, \dotsc\}$ is a collection of elements of $A^3$ which satisfies conditions (1), (2), (3), and (4) of Definition~\ref{def:absolute} for some incidence theorem. Then for every prime ideal $\mathfrak{p}$ of $A$, the determinant corresponding to the conclusion vanishes in $A/\mathfrak{p}$ by Remark~\ref{rem:integral}. This implies that the conclusion determinant is nilpotent. In particular, the incidence theorem holds in $A$ if $A$ is reduced. 
\end{remark}

\begin{remark}
If $A$ is a commutative ring, then the set of columns of nonzero $3 \times 3$ minors of a matrix with coefficients in $A$ need not be the bases of a rank $3$ matroid. For example, consider the following matrix, with coefficients in $\mathbb{Q}[\varepsilon]/(\varepsilon^2)$:
$$\begin{pmatrix} 1 & 0 & 0 & 0 & 0 \\ 0 & \varepsilon & 1 & 0 & 5 \\ 0 & 0 & \varepsilon & \varepsilon & 3 \end{pmatrix}.$$
Then the columns of nonzero $3 \times 3$ minors are the sets $\{1, 2, 5\}, \{1, 3, 4\}, \{1, 3, 5\}, \{1, 4, 5\}$, which is not the set of bases of a matroid. 
\end{remark}

A collection of nondegeneracy conditions defines a Zariski open subset of the space of $3 \times n$ matrices $\mathbb{A}^{3n}_{\mathbb{Z}}$, and a collection of collinearity conditions defines a closed subscheme of $\mathbb{A}^{3n}_{\mathbb{Z}}$. The corresponding absolute incidence theorem holds if and only if the intersection of these subschemes is equal to the subscheme where the determinant of the $3 \times 3$ minor corresponding to the conclusion also vanishes. The incidence theorem holds if and only if these subschemes of $\mathbb{A}^{3n}_{\mathbb{Z}}$ have the same underlying set. 

Recall that an Artinian local ring is a commutative ring with a unique maximal ideal which satisfies the descending chain condition on ideals. 

\begin{proposition}\label{prop:artinian}
An incidence theorem is an absolute incidence theorem if it holds for every $3 \times n$ matrix with entries in $A$ for every Artinian local ring $A$. 
\end{proposition}

\begin{proof}
Let $Z$ be the subscheme of $\mathbb{A}^{3n}_{\mathbb{Z}}$ obtained by intersecting the open set where the nondegeneracy conditions hold with the closed subscheme defined by the collinearity conditions. Let $Z_0$ be the intersection of $Z$ with the locus where the $3 \times 3$ minor corresponding to the conclusion vanishes. This is an absolute incidence theorem if and only if the inclusion $Z_0 \subseteq Z$ is an equality. 

To check whether this inclusion is an equality, we can choose an affine open cover of $Z$ and check it on each affine piece. The claim then follows from the following statement: let $R$ be a Noetherian ring, and let $I$ be a nonzero proper ideal of $R$. Then there is an Artinian local ring $A$ and a ring homomorphism $\varphi \colon R \to A$ such that $\varphi(I) \not=0$. 

To prove this, we first choose some nonzero $f \in I$. Then $\operatorname{ann}(f) = \{g \in R : gf = 0\}$ is a proper ideal. Choose a maximal ideal $\mathfrak{m}$ containing $\operatorname{ann}(f)$. Then the image of $f$ in $R_{\mathfrak{m}}$ is nonzero, so the image of $I$ in $R_{\mathfrak{m}}$ is nonzero. As $R_{\mathfrak{m}}$ is a Noetherian local ring, $\cap_k \mathfrak{m}^k = 0$, so there is some $j$ such that $I \not \subset \mathfrak{m}^j$. Then the image of $I$ in $R_{\mathfrak{m}}/\mathfrak{m}^j$ is nonzero, and $R_{\mathfrak{m}}/\mathfrak{m}^j$ is an Artinian local ring. 
\end{proof}

Most algebraic techniques that are used to prove incidence theorems can be modified to prove the corresponding absolute incidence theorem. The main subtleties come from the failure of the zero-product property for commutative rings. As an illustration, we give a direct algebraic proof that Pappus's theorem (Example~\ref{ex:pappus}) is an absolute incidence theorem. 
Given elements $x, y, z$ in $A^3$, we set $[x, y, z]$ to be the determinant of the $3 \times 3$ matrix with columns given by $x$, $y$, and $z$.

\begin{example}
Let $x_1, x_2, \dotsc, x_9$ be the columns of a $3 \times 9$ matrix of elements of $A$ satisfying the hypothesis of Pappus's theorem. The nondegeneracy conditions guarantee that $[x_1, x_2, x_5]$, $[x_1, x_4, x_5]$, and $[x_2, x_4, x_5]$ are all units, so, by Cramer's rule, we can find units $\lambda_1, \lambda_2$, and $\lambda_3$ such that $x_5 = \lambda_1 x_1 + \lambda_2 x_2 + \lambda_3 x_4$. Let $B$ be the $3 \times 3$ matrix with columns $\lambda_1 x_1, \lambda_2 x_2$, and $\lambda_3 x_4$. The nondegeneracy conditions imply that the determinant of $B$ is a unit, so $B$ is invertible. By multiplying the $3 \times 9$ matrix by $B^{-1}$ and scaling some columns by units, we reduce to the case when $x_1 = (1, 0, 0)$, $x_2 = (0, 1, 0)$, $x_4 = (0, 0, 1)$, and $x_5 = (1,1,1)$. Using the collinearity conditions, we deduce the following form of the matrix: 
$$S = \begin{pmatrix} 1 & 0 & ? & 0 & 1 & ? & 0 & ? & ? \\ 0 & 1 & ? & 0 & 1 & ? & ? & ? & ?  \\ 0 & 0 & 0 & 1 & 1 & ? & ? & ? & ?\end{pmatrix}.$$
The nondegeneracy conditions guarantee that every entry marked with a $?$ is a unit in $A$, and so by scaling the columns we can assume that the first nonzero entry in each column is $1$. As $[x_1, x_5, x_7] = [x_4, x_5, x_6] = 0$, we deduce that $S_{3,7} = S_{2,6} = 1$. Set $S_{2,3} = a, S_{3,6} = b,$ and $S_{2,9} = c$. Using the collinearity conditions $\{2, 6, 9\}, \{3, 4, 8\}$, and $\{1, 6, 8\}$, we deduce the following form of the matrix:
$$S = \begin{pmatrix} 1 & 0 & 1 & 0 & 1 & 1 & 0 & 1 & 1 \\ 0 & 1 & a & 0 & 1 & 1 & 1 & a & c  \\ 0 & 0 & 0 & 1 & 1 & b & 1 & ab & b\end{pmatrix}.$$
Then $[x_3, x_5, x_9] = a + b - ab - c = -[x_7, x_8, x_9]$, so $[x_7, x_8, x_9] = 0$ as $\{3, 5, 9\}$ is a collinearity condition. 
\end{example}

\section{The master theorem proves absolute incidence theorems}\label{sec:master_theorem}

As in the previous section, we fix a commutative ring $A$. For an element $x \in A^3$, we refer to its coordinates as $x_1, x_2, x_3$. Given elements $x$ and $y$ in $A^3$, we set $x \cdot y = x_1 y_1 + x_2 y_2 + x_3 y_3$ and $x \times y = (x_2 y_3 - x_3y_2, x_3y_1 - x_1y_3, x_1y_2 - x_2y_1)$. 
Recall that $[x, y, z] = (x \times y) \cdot z$, where $[x, y, z]$ is the determinant of the $3 \times 3$ matrix with columns $x, y$, and $z$.  

\begin{lemma}\label{lem:cross}
Let $v, s$, and $t$ be elements of $A^3$. Suppose that $v \cdot s = v \cdot t = 0$, and the ideal generated by the entries of $s \times t$ is the unit ideal. Then $v = \lambda (s \times t)$ for some $\lambda \in A$. 
\end{lemma}

\begin{proof}
Set $s \times t = (c_1, c_2, c_3)$. By multiplying the equation $v_1 s_1 + v_2 s_2 + v_3 s_3 = 0$ by $t_3$, multiplying the equation $v_1 t_1 + v_2 t_2 + v_3 t_3 = 0$ by $s_3$, and subtracting, we deduce that $v_1 c_2 = v_2 c_1$. Similarly, we have that $v_1 c_3 = v_3 c_1$ and $v_2 c_3 = v_3 c_2$. 

Because the ideal generated by $c_1, c_2, c_3$ is the unit ideal, there are $a_1, a_2, a_3 \in A$ such that $a_1 c_1 + a_2 c_2 + a_3 c_3 = 1$. Set $\lambda = a_1 v_1 + a_2 v_2 + a_3 v_3$. We claim that $\lambda c_i = v_i$ for each $i$. For example,
\begin{equation*}
\lambda c_1 = a_1 v_1 c_1 + a_2 v_2 c_1 + a_3 v_3 c_1 = a_1 v_1 c_1 + a_2 v_1 c_2 + a_3 v_1 c_3 = v_1(a_1 c_1 + a_2 c_2 + a_3 c_3) = v_1. \qedhere \end{equation*}
\end{proof}

\begin{lemma}\label{lem:parallel}
Let $v, w$ be elements of $A^3$, and suppose that the ideal generated by the entries of $v \times w$ is the unit ideal. Then for any unit $u \in A$ and any $\lambda \in A$, the ideal generated by the entries of $\lambda v + u w$ is the unit ideal. 
\end{lemma}

\begin{proof}
Suppose that the ideal generated by the entries of $\lambda v + u w$ is a proper ideal. Then it is contained in some maximal ideal $\mathfrak{m}$ of $A$. Consider the images of $v$ and $w$ in the vector space $(A/\mathfrak{m})^3$. The images are not parallel, and the image of $u$ in $A/\mathfrak{m}$ is nonzero, giving a contradiction. 
\end{proof}

\begin{lemma}\label{lem:crossnotparallel}
Let $v, w, s, t$ be elements of $A^3$. Suppose that $[s, t, v]$ is a unit, and that the ideal generated by the entries of $v \times w$ is the unit ideal. Then the ideal generated by the entries of $(s \times t) \times (v \times w)$ is the unit ideal. 
\end{lemma}

\begin{proof}
By the triple product identity, we have $(s \times t) \times (v \times w) = ((s \times t) \cdot w)v - ((s \times t) \cdot v)w = [s, t, w]v - [s, t, v]w$. As $[s, t, v]$ is a unit, the result follows from Lemma~\ref{lem:parallel}. 
\end{proof}

\begin{lemma}\label{lem:identity}
For any $s, t, v, w, P$ and $Q$ in $A^3$, we have
$$[s, t, P] [v, w, Q] - [v, w, P] [s, t, Q] = [P, Q, (s \times t) \times (v \times w)]. $$
\end{lemma}

\begin{proof}
It suffices to check this when the coordinates of $s, t, v, w, P$ and $Q$ are algebraically independent, so we may take $A = \mathbb{Q}(x_1, \dotsc, x_{18})$ and $s = (x_1, x_2, x_3), t = (x_4, x_5, x_6)$, and so on. If we fix $s, t, v$, and $w$, then both sides are bilinear functions of $P$ and $Q$. It therefore suffices to check this identity when $P$ and $Q$ are standard basis vectors, when it is obvious. 
\end{proof}

The identity in Lemma~\ref{lem:identity} is closely related to the mixed cross-ratio discussed in \cite[Definition 2.4]{FominPylyavskyy}, which is used in Fomin and Pylyavskyy's proof of Theorem~\ref{thm:master}. 

\begin{proposition}\label{prop:fundamental}
Let $A$ be a commutative ring, and let $s, t, v, w, P, Q, R$ be elements of $A^3$. Suppose that both the ideal generated by the entries of $v \times w$ and the ideal generated by the entries of $R$ are the unit ideal, and that $[s, t, v]$ is a unit. Assume that $[s, t, R] = [v, w, R] = 0$. Then there is a unit $u$ of $A$ such that
$$[s, t, P] [v, w, Q] - [v, w, P] [s, t, Q] = u[P, Q, R].$$ 
\end{proposition}

\begin{proof}
As $[s, t, v]$ is a unit, Lemma~\ref{lem:crossnotparallel} implies that the ideal generated by the entries of $(s \times t) \times (v \times w)$ is the unit ideal. 
The assumptions state that $(s \times t) \cdot R = (v \times w) \cdot R = 0$, so Lemma~\ref{lem:cross} implies that $R = \lambda (s \times t) \times (v \times w)$ for some $\lambda \in A$. The ideal generated by the entries of $R$ is therefore contained in the ideal $(\lambda)$, so $\lambda$ is a unit. The result then follows from Lemma~\ref{lem:identity}, with $u = \lambda^{-1}$. 
\end{proof}

\begin{proof}[Proof of Theorem~\ref{thm:absolutemaster}]
Consider a tile where the white vertices are not equivalent. Suppose that the white vertices correspond to the points $s_i, t_i$ and $s_j, t_j$, respectively, the black vertices correspond to points $P, Q$, and the point $R_k$ is the point associated to the tile. The hypotheses of Proposition~\ref{prop:fundamental} are satisfied, and so there is a unit $u \in A$ such that 
\begin{equation}\label{eq:fundamental}
[s_i, t_i, P] [s_j, t_j, Q] - [s_j, t_j, P] [s_i, t_i, Q] = u[P, Q, R_k].
\end{equation}
In particular, $[P, Q, R_k]$ is equal to $0$ if and only if the left-hand side is equal to $0$. If the tile is not the tile corresponding to the conclusion, then the collinearity conditions state that $[P, Q, R_k] = 0$, and so $[s_i, t_i, P] [s_j, t_j, Q] = [s_j, t_j, P] [s_i, t_i, Q]$. If the white vertices of the tile are equivalent, then $s_i = s_j$ and $t_i = t_j$, so this identity is automatic. 

The orientation of the surface induces an orientation of each tile; this is a direction of the cycles that make up the boundary of the tile, such that if two tiles share an edge, then the orientation of those two edges is opposite. In each tile, there are two edges which point from black vertices to white vertices and two edges which point from white vertices to black vertices. For each tile which does not correspond to the conclusion, we have an equation $[s_i, t_i, P] [s_j, t_j, Q] = [s_j, t_j, P] [s_i, t_i, Q]$, where the terms corresponding to the edges pointing from black vertices to white vertices are on the left. We multiply all of these equations together. 
The nondegeneracy conditions state that $[s_i, t_i, P]$ is a unit for each tile. For each edge which does not appear in the conclusion, the corresponding term will appear on both the left-hand side and the right-hand side. We can therefore cancel each of these, showing that the left-hand side of the instance of \eqref{eq:fundamental} corresponding to the conclusion vanishes. Because $u$ is a unit, this implies the conclusion. 
\end{proof}

\begin{example}\label{ex:pappus}
We now describe how to prove that Pappus's theorem is an absolute incidence theorem using Theorem~\ref{thm:absolutemaster}, following \cite[Proof of Theorem 3.2]{FominPylyavskyy}. To do this, we will prove an absolute incidence theorem involving $15$ points (labeled $P_1, P_2, \dotsc, P_6, R_1, R_2, R_3, s_1, s_2, s_3, t_1, t_2, t_3$, see Figure~\ref{fig:modifiedpappus}) which is a mild strengthening of Pappus's theorem. We impose the nondegeneracy conditions $\{s_i, t_i\}$ for each $i$ as well as $\{s_i, t_i, t_j\}$ and $\{s_i, s_j, t_j\}$ for each $i\not=j$. Let $\ell_i$ be the line spanned by $s_i$ and $t_i$. We apply Theorem~\ref{thm:absolutemaster} to the tiling of the torus in Figure~\ref{fig:pappustiling} with $9$ tiles, $3$ white vertices, and $6$ black vertices, where the tile corresponding to the conclusion has vertices labeled $\ell_1, \ell_3, P_2$, and $P_5$. This implies that if we impose the collinearity conditions $\{s_i, t_i, R_j\}$ for each $j \not= i$ and $\{P_i, P_j, R_k\}$ for each line appearing in Figure~\ref{fig:modifiedpappus} passing through the $P_i$, then the collinearity condition $\{P_2, P_5, R_2\}$ holds. 

In order to use this to deduce the classical statement of Pappus's theorem  (which involves only the $9$ points $P_1, \dotsc, P_6, R_1, R_2, R_3$), we can consider the coordinate projection of the subscheme of the space of $3 \times 15$ matrices to the space of $3 \times 9$ matrices, forgetting the columns labeled by $s_1, s_2, s_3, t_1, t_2, t_3$. As $s_1$ and $t_1$ are general points on the line spanned by $R_2$ and $R_3$ (and similarly for $s_2, t_2$ and $s_3, t_3$), the existence of these columns imposes no conditions on the remaining columns. See Remark~\ref{rem:aux}. 
\end{example}

\begin{figure}
		\begin{tikzpicture}
			
			\coordinate (A) at (0,0);
			\coordinate (B) at (2,0);
			\coordinate (C) at (4,0);
			
			\coordinate (D) at (1,2);
			\coordinate (E) at (2.5,2);
			\coordinate (F) at (5,2);
			
			\draw[thick] (A) -- (C); 
			\draw[thick] (D) -- (F);
			
			\draw[thick] (A) -- (E);
			\draw[thick] (A) -- (F);
			\draw[thick] (B) -- (D);
			\draw[thick] (B) -- (F);
			\draw[thick] (C) -- (D);
			\draw[thick] (C) -- (E);
			
			\coordinate (P) at (intersection of A--E and B--D);
			\coordinate (Q) at (intersection of A--F and C--D);
			\coordinate (R) at (intersection of B--F and C--E);
			
			\fill[black] (P) circle (1.5pt) node[label={[xshift=-0.4cm, yshift=-0.4cm]$P_2$}] {};
			\fill[black] (Q) circle (1.5pt) node[label={[xshift=-0.5cm, yshift=-0.3cm]$R_2$}] {};
			\fill[black] (R) circle (1.5pt) node[label={[xshift=0.4cm, yshift=-0.38cm]$P_5$}] {};

			\fill[black] (A) circle (1.5pt) node[below] {$P_3$};
			\fill[black] (B) circle (1.5pt) node[label={[xshift=0.2cm, yshift=-0.6cm]$R_3$}] {};
			\fill[black] (C) circle (1.5pt) node[below] {$P_6$};
			
			\fill[black] (D) circle (1.5pt) node[above] {$P_1$};
			\fill[black] (E) circle (1.5pt) node[above right] {$R_1$};
			\fill[black] (F) circle (1.5pt) node[above] {$P_4$};
			\draw (B) -- ($(E)!3cm!(B)$);
			\draw (B) -- ($(B)!3cm!(E)$);
			\draw (B) -- ($(Q)!2cm!(B)$);
			\draw (B) -- ($(B)!3cm!(Q)$);
			\draw (E) -- ($(Q)!2cm!(E)$);
			\draw (E) -- ($(E)!3cm!(Q)$);
			\coordinate (s2) at ($(E)!3cm!(B)$);
			\coordinate (t2) at ($(B)!3cm!(E)$);
			\coordinate (t1) at ($(B)!3cm!(Q)$);
			\coordinate (s1) at ($(Q)!2cm!(B)$);
			\coordinate (t3) at ($(Q)!2cm!(E)$);
			\coordinate (s3) at ($(E)!3cm!(Q)$);

			\fill[black] (s2) circle (1.5pt) node[below] {$s_2$};
			\fill[black] (t2) circle (1.5pt) node[above] {$t_2$};
			\fill[black] (t1) circle (1.5pt) node[above] {$t_1$};
			\fill[black] (s1) circle (1.5pt) node[left] {$s_1$};
			\fill[black] (t3) circle (1.5pt) node[left] {$t_3$};
			\fill[black] (s3) circle (1.5pt) node[right] {$s_3$};
		\end{tikzpicture}\caption{The three lines added to the Pappus configuration.}\label{fig:modifiedpappus}
\end{figure}
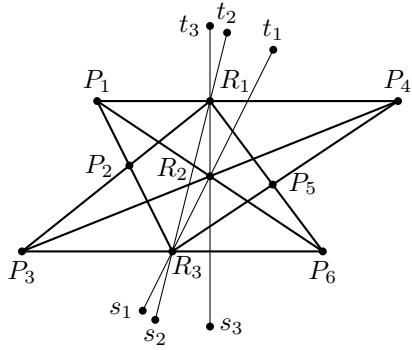

\begin{figure}
\begin{tikzpicture}[
    scale=1.5,
    thick,
    wdot/.style={circle, draw=black, fill=white, inner sep=0pt, minimum size=6pt},
    bdot/.style={circle, draw=black, fill=black, inner sep=0pt, minimum size=6pt}
]

    \node[wdot] (a) at (0, 2) {$\ell_1$};
    \node[wdot] (b) at (-1.5, 0) {$\ell_1$};
    \node[wdot] (c) at (1.5, 0) {$\ell_1$};
    \node[wdot] (d) at (1.5, 1) {$\ell_2$};
    \node[wdot] (e) at (-1.5, 1) {$\ell_2$};
    \node[wdot] (f) at (0, -1) {$\ell_2$};
    \node[wdot] (g) at (0, 0.5)  {$\ell_3$};

    \node[bdot] (P1) at (-0.5, 1) {};
    \node[bdot] (P2) at (0.5, 1) {};
    \node[bdot] (P3) at (1, 0.5) {};
    \node[bdot] (P4) at (0.5, 0) {};
    \node[bdot] (P5) at (-0.5, 0) {};
    \node[bdot] (P6) at (-1, 0.5) {}; 

    \draw (a) -- (d) -- (c) -- (f) -- (b) -- (e) -- (a) [dashed];
    \draw (g) -- (P1) -- (a) -- (P2) -- (g);
    \draw (g) -- (P2) -- (d) -- (P3) -- (g);
    \draw (g) -- (P3) -- (c) -- (P4) -- (g);
    \draw (g) -- (P4) -- (f) -- (P5) -- (g);
    \draw (g) -- (P5) -- (b) -- (P6) -- (g);
    \draw (g) -- (P6) -- (e) -- (P1) -- (g);
    \node[above right] at (0.5, 1) {$P_6$};
    \node[right] at (1, 0.5) {$P_5$};
    \node[left] at (0.5, 0) {$P_2$};
    \node[right] at (-0.5, 0) {$P_3$};
    \node[label={[xshift=-0.4cm, yshift=-0.4cm]$P_4$}] at (-1, 0.5) {};
    \node[above left] at (-0.5, 1) {$P_1$};

\end{tikzpicture}\caption{The tiling used to prove Pappus's theorem. Opposite sides of the hexagon are glued; there are no lines between the outer white vertices.}\label{fig:pappustiling}
\end{figure}
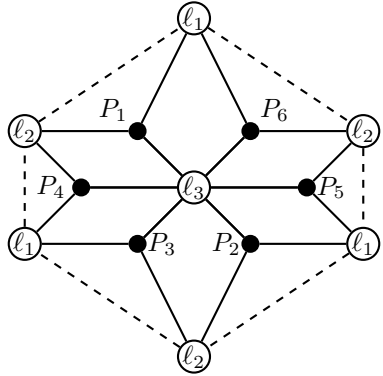

\begin{remark}\label{rem:extra}
Except in some degenerate cases, it is possible to eliminate the $s_i$ and $t_i$ points in the absolute incidence theorems proved by the master theorem, by adding the conditions that if $R_i, R_j,$ and $R_k$ are all vertices associated to tiles involving one of the equivalence classes of white vertices, then $R_i, R_j$, and $R_k$ are collinear. 
\end{remark}

\begin{remark}\label{rem:aux}
Whether or not an incidence theorem is absolute is not changed by ``auxiliary constructions,'' in the following sense. Consider an incidence theorem involving $n$ points with the conclusion that $p_a, p_b$, and $p_c$ are collinear for some $a, b, c \in \{1, \dotsc, n\}$. The incidence theorem defines a subscheme $Z$ of $\mathbb{A}^{3n}_{\mathbb{Z}}$ on which $[a, b, c]$ vanishes at every closed point. Suppose that, for some $m \ge n$, we have an incidence theorem involving $m$ points with the conclusion that $p_a, p_b$, and $p_c$ are collinear. This incidence theorem defines a subscheme $\tilde{Z}$ of $\mathbb{A}^{3m}_{\mathbb{Z}}$. Suppose that the natural projection sends $\tilde{Z}$ to $Z$ and is scheme-theoretically surjective; this implies that a function on $Z$ vanishes if and only if its pullback to $\tilde{Z}$ vanishes. Then the incidence theorem on $n$ points is absolute if and only if the incidence theorem on $m$ points is absolute. For example, this situation arises if the incidence theorem on $m$ points is obtained by adding points which are contained in a certain line and not allowed to coincide with certain other points. 
\end{remark}

\begin{remark}
While Theorem~\ref{thm:absolutemaster} shows that every instance of the master theorem is an absolute incidence theorem, it does not rule out the possibility of using the master theorem together with case analysis or proof by contradiction to deduce an incidence theorem which is not an absolute incidence theorem. 
See \cite[Section 2.3]{PS} for a discussion. 
\end{remark}

\section{Two proofs of Theorem~\ref{thm:incidencestrong}}\label{sec:incidenceproof}

\begin{proof}[First proof of Theorem~\ref{thm:incidencestrong}]
	We present a computer-aided proof of Theorem~\ref{thm:incidencestrong} using \texttt{OSCAR}~\cite{OSCAR-book}.\footnote{We used \texttt{OSCAR} 1.5.0 in \texttt{julia} 1.10.2. All computations terminated within one minute on our laptop.}
	The proof distinguishes the two cases of whether the points  $\{p_1,p_3,p_6\}$ are collinear or not.
	First assume that the points $\{p_1,p_3,p_6\}$ are not collinear.
	
	Any choice of $13$ points in the projective plane over $\mathbf{k}$ can be represented by a $3\times 13$ matrix with entries in $\mathbf{k}$, where the $i$th column is the vector of projective coordinates of $p_i$.
	We parametrize all such configurations by a matrix $A$ where each entry is a new variable in the polynomial ring over $\mathbb{Z}$, as these variables can then be specialized to any field.

	Using the nondegeneracy and collinearity conditions, we can simplify this matrix:
	\begin{itemize}
		\item The assumption that the points $\{p_1,p_3,p_6\}$ are not collinear allows us to assume that the minor on these columns is the identity matrix.
		\item The assumption that the points $\{p_1,p_2,p_3\}$ and $\{p_1,p_3,p_{13}\}$ are collinear now implies  $A_{3,2}=A_{3,13}=0$. 
		\item Similarly, the assumption that $\{p_3,p_5,p_6\}$ and $\{p_1,p_6,p_{9}\}$ are collinear now implies that $A_{1,5}=A_{2,9}=0$. 
		\item If the entry $A_{1,2}$ is zero, then the second and the third columns would be parallel, contradicting our assumption that all points are pairwise distinct. Hence we obtain $A_{1,2}\neq 0$. After scaling the second column, we can therefore assume that $A_{1,2}=1$.
		By the same argument, we can assume that $A_{2,5}=A_{1,9}=A_{1,13}=1$.
		\item Similarly, we obtain that $A_{2,2}\neq 0$.
		Scaling a row of $A$ by a nonzero scalar in $\mathbf{k}$ does not change the (non)incidence relations between the chosen points.
		Therefore, we can scale the second row so that $A_{2,2}=1$. After rescaling the third and fifth columns, we still have $A_{2,3}=A_{2,5}=1$.
		Using the analogous argument on the fifth column yields $A_{3,5}=1$.
	\end{itemize}
	Hence we can assume that 
	\[
	A = 
	\begin{pmatrix}
		1  & 1    	 &   0   & x_1   &   0     &  0    &   x_4        &  x_{7} &     1           &  x_{11}    &  x_{14}  &  x_{17} &  1          \\
		0  & 1       &  1    & x_2   &   1      &  0    &  x_5        & x_{8}  &    0            &  x_{12}   &   x_{15} &  x_{18}  & x_{20}\\
		0  & 0		 &   0   & x_3   &   1      & 1     &  x_{6}     & x_{9}  &  x_{10}  &  x_{13}   &   x_{16}  & x_{19}  &   0
	\end{pmatrix}.
	\]
	Let $I_{\text{minors}}$ be the ideal in $\mathbb{Z}[x_1,\dots,x_{20}]$ of the minors of $A$ corresponding to the remaining $14$ nontrivial collinearity conditions.
	The Gr\"obner basis of  $I_{\text{minors}}$ computed in \texttt{OSCAR} has $145$ generators:
\begin{minted}{jlcon}
Julia> using Oscar

julia> Collinearities = [[1, 4, 5],  [1, 7, 10], [1, 8, 12], [2, 4, 6], [2, 5, 8], [2, 10, 11], [3, 4, 7], 
           [3, 8, 10], [4, 9, 10], [5, 7, 11], [6, 7, 13], [6, 8, 11], [6, 10, 12], [7, 9, 12]];

julia> S, x = polynomial_ring(ZZ, 20);



julia> A = matrix(S, [1    1   0   x[1]    0   0    x[4]   x[7]     1   x[11]   x[14]   x[17]     1;
                      0    1   1   x[2]    1   0    x[5]   x[8]     0   x[12]   x[15]   x[18]   x[20];
                      0    0   0   x[3]    1   1    x[6]   x[9]   x[10] x[13]   x[16]   x[19]     0]);

julia> Iminors = ideal(S,[det(A[:,col]) for col in Collinearities]);

julia> length(groebner_basis(Iminors))
145
\end{minted}
	
	Now we saturate the ideal $I_{\text{minors}}$ with respect to the $\binom{13}{2}$ ideals of the three $2\times 2$ minors of every pair of columns of $A$. This removes the components where two columns are parallel. 
	Call the resulting ideal $I_{\text{sat}}$; it has a Gr\"obner basis with $203$ generators.
	The Gr\"obner basis of the radical $\sqrt{I_{\text{sat}}}$ has $75$ generators:
\begin{minted}{jlcon}
julia> Isat = Iminors;

julia> for pair in combinations(1:13,2)
           next_col_ideal = ideal(S,[det(A[rows,pair]) for rows in [[1,2],[1,3],[2,3]]]);
           Isat = saturation(Isat,next_col_ideal);
       end

julia> length(groebner_basis(Isat))
203

julia> length(groebner_basis(radical(Isat)))
75
\end{minted}
	
	Let $A[11,12,13]$ be the minor of $A$ on the last three columns.
	Another \texttt{OSCAR} computation yields
	\[
	\det(A[11,12,13]) \notin I_{\text{sat}} \text{ and } \det(A[11,12,13]) \in \sqrt{I_{\text{sat}}}:
	\]
\begin{minted}{jlcon}
julia> det(A[:,[11,12,13]]) in Isat
false

julia> det(A[:,[11,12,13]]) in radical(Isat)
true
\end{minted}
	This confirms that under the assumption that the points  $\{p_1,p_3,p_6\}$ are not collinear, the given incidence relations in Theorem~\ref{thm:incidencestrong} form an incidence theorem which is not an absolute incidence theorem.
	
	For the second case, we assume that the points $\{p_1,p_3,p_6\}$ lie on a line $L$.
	We claim that this already implies that~$L$ contains all $13$ points $p_1,\dots,p_{13}$, so the conclusion of the incidence theorem holds trivially:
	Since, by assumption, the points $\{p_1, p_2, p_3\},   \{p_3, p_5, p_6\},  \{p_1, p_6, p_9\},\{p_1, p_3, p_{13}\}$ are collinear, the line $L$ also contains the points $p_2$, $p_5$, $p_9$ and $p_{13}$.
	Similarly, since the points $\{p_1, p_4, p_5\},  \{p_6, p_{7}, p_{13}\}, \{p_2, p_5, p_8\}$ are collinear, the line~$L$ contains the points $p_4$, $ p_7$ and $p_8$.
	Lastly, since the points $\{p_1, p_7, p_{10}\},  \{p_5, p_{7}, p_{11}\}, \{p_1, p_8, p_{12}\}$ are collinear, the line~$L$ contains the remaining points $p_{10}$,  $p_{11}$ and $p_{12}$.
\end{proof}

\begin{proof}[Second proof of Theorem~\ref{thm:incidencestrong}]
First consider the case when the following triples of points are \emph{not} collinear:
\begin{equation}\begin{split}\label{eq:nondegen}
&\{p_1,p_3,p_4\}, \{p_1,p_3,p_6\}, \{p_1,p_4,p_6\}, \{p_3,p_4,p_6\}, \{p_3,p_6,p_i\}_{7\le i\le 13}.
\end{split}\end{equation}
Using a projective transformation, because none of $\{p_1,p_3,p_4\}, \{p_1,p_3,p_6\}, \{p_1,p_4,p_6\}, \{p_3,p_4,p_6\}$ are collinear, we can assume that $p_1 = [1 : 0 : 0]$, $p_3 = [0 : 1 : 0]$, $p_4 = [1 : 1 : 1]$, and $p_6 = [0 : 0 : 1]$. The assumption that, for each $i \ge 7$, $\{p_3, p_6, p_i\}$ is not collinear means that the first coordinate of $p_i$ is nonzero for each $i \ge 7$. We can then rescale each of these so that the first coordinate is $1$.  Each collinearity involving two elements of $\{p_1, p_3, p_6\}$ implies that some coordinate is $0$.
The collinearities $\{p_1, p_4, p_5\}, \{p_3,p_4,p_7\}$, and $\{p_2, p_4, p_6\}$ imply that the second coordinate of $p_2$, second coordinate of $p_5$, and third coordinate of $p_7$ are $1$. 
After rescaling, we deduce that the $3 \times 13$ matrix representing the point configuration looks like the following
$$A = \begin{pmatrix} 1 & 1 & 0 & 1 & 0 & 0 & 1 & 1 & 1 & 1 & 1 & 1 & 1 \\ 
0 & 1 & 1 & 1 & 1 & 0 & ? & ? & 0 & ? & ? & ? & ? \\ 
0 & 0 & 0 & 1 & 1 & 1 & 1 & ? & ? & ? & ? & ? & 0 \end{pmatrix}.$$
Set $x = A_{2,7}$, $y = A_{3,8}$, $z = A_{3,9}$, and $w = A_{3,12}$. Using the collinearities 
$$\{p_1, p_7, p_{10}\}, \{p_2, p_5, p_8\},\{p_3, p_8, p_{10}\}, \{p_5, p_7, p_{11}\}, \{p_6, p_{7}, p_{13}\},\{p_6, p_8, p_{11}\}, \{p_6, p_{10}, p_{12}\}, $$
one easily deduces the following relations between the entries of $A$. 
\begin{equation}\label{eq:easy}
A = \begin{pmatrix} 1 & 1 & 0 & 1 & 0 & 0 & 1 & 1 & 1 & 1 & 1 & 1 & 1 \\ 
0 & 1 & 1 & 1 & 1 & 0 & x & y+1 & 0 & xy & y+1 & xy & x \\ 
0 & 0 & 0 & 1 & 1 & 1 & 1 & y & z & y & y + 2 - x & w & 0 \end{pmatrix}.
\end{equation}
For example, the collinearity $\{p_3, p_8, p_{10}\}$ implies that $A_{3,10} = y$, and the collinearity $\{p_1, p_7, p_{10}\}$ implies that $A_{2,10} = xy$. There are now only four collinearities which impose nontrivial conditions on $A$. They are:
\begin{equation}\label{eq:21011}
\{p_2, p_{10}, p_{11}\}: xy^2 + 2xy - x^2y - y^2 - y + x - 2 = 0,
\end{equation}
\begin{equation}\label{eq:1812}
\{p_1, p_8, p_{12}\}: yw + w = xy^2,
\end{equation}
\begin{equation}\label{eq:4910}
\{p_4, p_9, p_{10}\}: xy + z - xyz - y = 0, \text{ and }
\end{equation}
\begin{equation}\label{eq:7912}
\{p_7, p_9, p_{12}\}: x(y + z - yz - w) = 0.
\end{equation}
Note that $x \not=0$, as $p_1$ is distinct from $p_{13}$. From \eqref{eq:7912},  we deduce that $w = y + z - yz$. Substituting  into \eqref{eq:1812} gives 
\begin{equation}\label{eq:solvez}
z (y^2 - 1) = y^2 + y - xy^2.
\end{equation}
If $y = -1$, then the equation corresponding to \eqref{eq:1812} implies that $x=0$, which is impossible. If $y=1$, then \eqref{eq:21011} becomes $-(x - 2)^2 = 0$. But if $x = 2$ and $y=1$, then $p_7 = p_8$. We can therefore assume that $y^2 - 1 \not=0$ and solve \eqref{eq:solvez} for $z$. We substitute this into \eqref{eq:4910} and clear denominators to obtain the equation
\begin{equation}\label{eq:xy}
y(2 - x + y - 2xy - y^2 + x^2y^2) = 0.
\end{equation}
We may assume that $y \not=0$, as otherwise we would have $p_2 = p_8$. We deduce that $2 - x + y - 2xy - y^2 + x^2y^2 = 0$. Adding this to \eqref{eq:21011}, we obtain
\begin{equation}\label{eq:simplexy}
y(-2y + x^2y + xy - x^2) = 0.
\end{equation}
As above, we may assume that $y\not=0$, so $x^2 = y(x+2)(x-1)$. If $x = 1$, then $p_4 = p_7$, and if $x = -2$, then we see that $x=0$, which is impossible. We can then solve for $y$ and substitute into \eqref{eq:21011}. After clearing denominators, we obtain
$$x^6 - 3 x^5 - 2 x^4 + 12 x^3 - 4 x^2 - 12 x + 8 = (x - 2) (x - 1) (x^2 - 2)^2 = 0.$$
If $x = 2$, then \eqref{eq:21011} implies that $y = 0$ or $y=1$, which contradicts that $p_7 \not= p_8$ or that $p_2 \not= p_8$. As above, $x = 1$ is impossible. Because $\mathbf{k}$ is a field, we deduce that $x = \pm \sqrt{2}$; if $\mathbf{k}$ has characteristic $2$, this means that $x=0$, which is impossible. Then \eqref{eq:simplexy} yields that $y = \pm \sqrt{2}$.  From \eqref{eq:4910}, we deduce that $z = 2 \mp \sqrt{2}$, so \eqref{eq:7912} implies that $w = 4 \mp 2\sqrt{2}$. Plugging this into \eqref{eq:easy}, we deduce the following form of $A$:
\begin{equation}\label{eq:formA}
A = \begin{pmatrix} 1 & 1 & 0 & 1 & 0 & 0 & 1 & 1 & 1 & 1 & 1 & 1 & 1 \\ 
0 & 1 & 1 & 1 & 1 & 0 & x & x+1 & 0 & 2 & x+1 & 2 & x \\ 
0 & 0 & 0 & 1 & 1 & 1 & 1 & x & 2-x & x & 2 & 4 - 2x & 0 \end{pmatrix},
\end{equation}
where $x = \pm \sqrt{2}$. We see that the minor with columns $11, 12, 13$ is $0$, so $p_{11}, p_{12},$ and $p_{13}$ are collinear. 

If any one of the $11$ triples in \eqref{eq:nondegen} is collinear, then one can check that the collinearity conditions force all $13$ points to lie on a line. The argument works analogously as in the second part of the first proof above. In particular, $\{p_{11}, p_{12}, p_{13}\}$ is collinear, so the incidence theorem holds. 

To check that this is not an absolute incidence theorem, it suffices to check that the conclusion of Theorem~\ref{thm:incidencestrong} does not hold over the ring $B = \mathbb{Q}[x, \varepsilon]/(\varepsilon^2, x^2 - 2 - \varepsilon/4)$. Consider the matrix from \eqref{eq:matrix}.
It is easy to check that if we reduce modulo the unique maximal ideal $(\varepsilon)$, then no column of the matrix is a scalar multiple of another, so Lemma~\ref{lem:field} shows that \eqref{eq:matrix} satisfies the nondegeneracy conditions. 
A somewhat tedious computation shows that the determinant of any $3 \times 3$ minor  corresponding to a triple of collinear points in Theorem~\ref{thm:incidencestrong} vanishes. It turns out that there are only five conditions to check which are not obvious: those corresponding to $\{p_1, p_7, p_{10}\}, \{p_1, p_8, p_{12}\}, \{p_2, p_{10}, p_{11}\}, \{p_4, p_9, p_{10}\}$, and $\{p_7, p_9, p_{12}\}$. 
However, the determinant of the minor with columns $11, 12, 13$ is $2 \varepsilon/7 - \varepsilon x/14 \not= 0$.
\end{proof}

\begin{remark}\label{rem:Nrealizable}
Let $N$ be the matroid obtained by relaxing $\{{11}, {12}, {13}\}$ in $M$, i.e., the non-bases of $N$ are exactly the collinearity conditions listed in Theorem~\ref{thm:incidencestrong}. The content of a slight weakening of Theorem~\ref{thm:incidencestrong} is that $N$ is not realizable over any field. We do not know if $N$ is representable over a skew field, is multilinear, or is algebraic.
It would be interesting but computationally challenging to check whether the tools recently introduced in~\cite{skewlinear} could certify that $N$ is not skew-linear, for instance by proving that $N$ does not admit a tensor product with $M(K_4)$.
 We have computed the foundation of $N$ (in the sense of \cite{BL}) using the program \cite{M2pasture}. The foundation $F_N$ can be described  as $\mathbb{F}_1^\pm(x_1,\dots,x_{22})$ modulo $451$ hexagons which generate the nullset of the pasture.
\end{remark}

\begin{remark}\label{rem:powerseries}
One way to attempt to produce a counterexample to Theorem~\ref{thm:incidencestrong} would be to choose a point configuration satisfying the hypothesis (such as \eqref{eq:formA}, with $x = \sqrt{2}$) and construct $39$ convergent power series $f_{1,1}(t), f_{1,2}(t), \dotsc, f_{3,13}(t)$ such that, if we consider the matrix $A(t)$ whose $(i,j)$ entry is $f_{i,j}(t)$, then $A(0) = A$ and $A(t)$ satisfies the collinearity conditions of Theorem~\ref{thm:incidencestrong}, but $[11, 12, 13]$ is nonzero. We can try to construct the power series term-by-term; the constant terms are determined by the condition that $A(0) = A$. It is possible to choose the linear terms of the $f_{i,j}$ such that the linear terms of the minors corresponding to the collinearity conditions in Theorem~\ref{thm:incidencestrong} vanish, but the linear term of $[11, 12, 13]$ is nonzero. However, it is not possible to then choose the quadratic terms of the power series so that the quadratic terms of the minors corresponding to the collinearity conditions are $0$. In this sense, the configuration in \eqref{eq:formA} admits an infinitesimal (or first order) motion that cannot be integrated. 
\end{remark}

\printbibliography

\end{document}